\documentclass[12pt,reqno]{amsart}
\usepackage{amsmath,amssymb,amscd,latexsym,amsthm,euscript}
\usepackage{mathrsfs}
\usepackage{comment}
\usepackage[unicode]{hyperref}
\usepackage{hypbmsec}
\usepackage{longtable}
\newtheorem{Thm}{Theorem}
\newtheorem{Con}{Conjecture}
\newtheorem{Lem}{Lemma}

\newtheorem{Def}{Definition}
\newtheorem{cor}{Corollary}
\newtheorem{Prop}{Proposition}
\theoremstyle{remark}

\renewcommand\d{{\mathrm d}}
\DeclareMathOperator{\ord}{ord}
\begin{document}

\title
{Forbidden integer ratios of consecutive power sums}
\author{Ioulia N. Baoulina and Pieter Moree}
\address{Department of Mathematics, Moscow State Pedagogical University, Krasnoprudnaya str. 14, Moscow 107140, Russia}
\email{jbaulina@mail.ru}
\address{Max-Planck-Institut f\"ur Mathematik, Vivatsgasse 7, D-53111 Bonn, Germany}
\email{moree@mpim-bonn.mpg.de}
\date{\today}
\dedicatory{To the memory of Prof. Wolfgang Schwarz }
\begin{abstract}
Let $S_k(m):=1^k+2^k+\cdots+(m-1)^k$ denote a power sum. In 2011 Bernd 
Kellner formulated
the conjecture that for $m\ge 4$ the ratio
$S_k(m+1)/S_k(m)$ of two consecutive power sums is never an
integer. We will develop some techniques that allow one to exclude many
integers $\rho$ as a ratio and combine them to exclude the integers
$3\le \rho\le 1501$ and, assuming a 
conjecture on irregular primes to be true, a set of density $1$ of ratios $\rho$. To exclude
a ratio $\rho$ one has to show that
the Erd\H{o}s-Moser type equation $(\rho-1)S_k(m)=m^k$ has no non-trivial solutions.
\end{abstract}

\maketitle

\section{Introduction}
\noindent Power sums have fascinated mathematicians for centuries. In this paper we consider
some Diophantine equations involving power sums, of which the Erd\H{o}s-Moser equation
\begin{equation}
\label{EME}
1^k+2^k+\cdots+(m-2)^k+(m-1)^k=m^k
\end{equation}
is typical and the most famous one. This equation has the obvious solution $(m,k)=(3,1)$ and
conjecturally no other solutions exist (this conjecture was formulated
around 1950 by Paul Erd\H{o}s in a letter to Leo Moser). 
Leo Moser \cite{Moser}, using only elementary number theory, established the 
following result. 
\begin{Thm}[{Leo Moser, 1953}]
\label{Leo}
If $(m,k)$ is a solution of \eqref{EME} with $k\ge 2$, then
$m>10^{10^6}$.
\end{Thm}
\noindent For the shortest proof of this result presently known, we refer to
Moree \cite{tophat}.

Using very different techniques, namely continued fractions and a many decimal computation of
$\log 2$, Gallot et al.~\cite{GMZ} established the current world record:
\begin{Thm}[{Gallot et al., 2011}]
\label{YVES}
If an integer pair $(m,k)$ with $k\ge 2$ satisfies \eqref{EME}, then
$$m > 2.7139 \cdot 10^{\,1\,667\,658\,416}.$$
\end{Thm}
The bound $m>10^{10^{10}}$ seems feasible, but requires somewhat better computer
resources than the authors of \cite{GMZ} had at their disposal.

Let $S_k(m):=\sum_{j=1}^{m-1}j^k$ be the sum of the first $m-1$ consecutive $k$th powers.
In this notation we can rewrite (\ref{EME}) as
\begin{equation}
\label{M}
S_k(m)=m^k.
\end{equation}

In the literature also the
{\it generalized Erd\H{o}s-Moser} conjecture is considered. 
The strongest result to date is due
to the second author~\cite{magics} who proved the following.
\begin{Thm}
\label{provo}
For a fixed positive integer $a$, the equation  $S_k(m)=am^k$ has no integer solutions $(m,k)$ with
$$k\ge 2,\quad m<\max \bigl(10^{9\cdot 10^6},a\cdot 10^{28}\bigr).$$
\end{Thm} 
Interestingly, the method of Gallot et al. allows one only to deal with
a specific value of $a$, for a general $a$ only the elementary method of Moser is available.

Kellner \cite{Kellner3} conjectured that if $k$ and $m$ are positive integers with $m\ge 3$, the
ratio $S_k(m+1)/S_k(m)$ is an integer iff $(m,k)\in \{(3,1),(3,3)\}$. Noting that
$S_k(m+1)=S_k(m)+m^k,$ one observes that this conjecture is equivalent with the
following one.
\begin{Con}[Kellner-Erd\H{o}s-Moser]
\label{kellie}  Let $m\ge 3$. 
We have  
\begin{equation}
\label{KEME}
aS_k(m)=m^k
\end{equation}
iff $(a,m,k)\in \{(1,3,1),(3,3,3)\}$.
\end{Con}
\noindent If this conjecture holds true, then obviously so does the Erd\H{o}s-Moser conjecture. However,
whereas the Erd\H{o}s-Moser conjecture is open, we are able to establish the unsolvability of (\ref{KEME}) 
for many integers $a$.
\begin{Thm}
\label{main}
If $a$ has a regular prime as divisor 
or $2\le a\le 1500$, then the equation $aS_k(m)=m^k$ 
has no solution with $m\ge 4$. 
\end{Thm}
\noindent (We consider $2$ to be a regular prime, see Section~\ref{preli} for details.)
The first restriction on $a$ is not very difficult to prove, but powerful in its consequences. 
Assume that there exists a real number
$\delta<1$ such that the number of irregular primes $p\le x$ is bounded above by $\delta x/\log x$ as
$x\rightarrow \infty$. It then follows (see Section~\ref{preli}) that, for a set of integers $a$ of density 1, $aS_k(m)=m^k$ 
has no solution with $m\ge 4$. 
The first restriction implies that in order to exclude the $a$ in the range $2\le a\le 1500$, one has to
exclude $a=37^2$ and all irregular primes in this interval. These 
$a$ can be 
dealt with using various technical and not very general necessary conditions for~\eqref{KEME} to be solvable.

In case we are not able to exclude a square-free $a$, we are able to show that if 
\eqref{KEME} holds, then both $k$ and $m$ are large.
\begin{Thm}
\label{main2} 
Suppose that $aS_k(m)=m^k$, $m\ge 4$ and $a$ is square-free, then both $k$ and $m$ exceed
$3.44\cdot 10^{82}$.
\end{Thm}
\noindent We like to point out that for solutions with $m\equiv 1\pmod{3}$ or $m\equiv 1\pmod{30}$ much
larger lower bounds hold true (see Theorem~\ref{thm:t6}).

For a fixed integer $a\ge 1$ it is not known whether there are finitely many solutions
$(m,k)$ of \eqref{KEME} or not. In this direction we can only contribute the following modest result.
\begin{Prop}
\label{unique}
Let $(m_1,k_1)$ and $(m_2,k_2)$ be different solutions of $aS_k(m)=m^k$.
Then $m_1\ne m_2$ and $k_1\ne k_2$.
\end{Prop}

A final approach of dealing with equation 
\eqref{KEME} is to try to prove that $k$ is divisible by $120$ say.
Once established, there are many options of how to get an even bigger number to divide $k$. A cascade
of ways how to proceed further arises and it seems very likely then that also in this case $a+1$ cannot
occur as a ratio. We demonstrate this cascade process in Section~\ref{sec:cascade}. Paul Tegelaar 
\cite{tilier} jokingly called this `the method of infinite ascent'.

After discussing some basic material on Bernoulli numbers 
and power sums in Section~\ref{preli}, we obtain a crucial result (Theorem~\ref{thm:t4}) relating 
Bernoulli numbers
and solutions of the Kellner-Erd\H{o}s-Moser equation in Section~\ref{sec:kell}. Integers $a$ that
are a product of irregular primes cannot be immediately excluded, and for these one can use 
\emph{helpful pairs}, see Section~\ref{sec:hp}. They allow one to rule out that $k\equiv c\pmod{d}$ for
many even integers $c\ge 2$ and $d$. In Section~\ref{exclusive} we demonstrate with both an easy and
a difficult example how to exclude a given integer ratio $\rho$. Table~\ref{table3} illustrates in a compact way
how to show that $3\le \rho\le 1501$ are forbidden ratios. Some ratios are clearly much easier to exclude
than others and this is discussed in Section \ref{sec:bad}.
In Section~\ref{sec:cascade} 
we discuss how to show that a given integer divides $k$.
In Section~\ref{sec:lower} we reason in the way of Moser to derive lower bounds for $k$ and $m$ in case $a$ is squarefree.
In Section~\ref{sec:proofs} proofs of the new results announced above are given. These 
are mainly based on work done in earlier sections.
In the final section, Section~\ref{sec:other}, some further properties of potential solutions of the Kellner-Erd\H{o}s-Moser
equation are derived. The proof that the integers $2\le a\le 1500$ are forbidden makes use of
Tables~\ref{table2} and \ref{table3}.

A survey of earlier work on Erd\H{o}s-Moser type conjectures can be found in
Moree~\cite{magics}, also see \cite[Chapter 8]{Borwein}, for an expository account of the
work of Gallot et al.~\cite{GMZ}.

\section{Preliminaries on Bernoulli numbers and power sums}
\label{preli}
\begin{Lem}[Carlitz-von Staudt]
\label{lem:staudt}
Let $k$ and $m$ be positive integers, then
$$
S_k(m)=\sum_{j=1}^{m-1}j^k=\begin{cases}
0\pmod{\frac{m(m-1)}2}&\text{if $k$ is odd},\\
-\sum_{p\mid m,(p-1)\mid k}{\frac mp}\pmod{m} &\text{otherwise}.
\end{cases}
$$
\end{Lem}
This result, with some small error (cf.~Moree~\cite{Canada}), was published in 1961 by Carlitz. For an
easy reproof of the above result, see Moree~\cite{tophat}.

Recall that the \emph{Bernoulli numbers} $B_k$ are defined by the power series
$$
\frac t{e^t-1}=\sum_{k=0}^{\infty}\frac{B_kt^k}{k!}.
$$ 
They are rational numbers and can be written
as $B_k=U_k/V_k$, with $V_k>0$ and $\gcd(U_k,V_k)=1$. 
One has $B_0=1$, $B_1=-1/2$, $B_2=1/6$ and $B_{2j+1}=0$ for $j\ge 1$.

In the next four lemmas, we record some well-known facts about the Bernoulli numbers (see \cite[Chapter~15]{IR}).

\begin{Lem}[von Staudt-Clausen]
\label{lem:clausen}
If $k\ge 2$ is even, then $V_k=\prod_{(p-1)\mid k}p$.
\end{Lem}

\begin{Lem}[Kummer congruence]
\label{lem:kummer}
Let $k\ge 2$ be even and $p$ be a prime with $(p-1)\nmid k$. If $k\equiv r\pmod{{p-1}}$, then $B_k/k\equiv B_r/r\pmod{p}$.
\end{Lem}

\begin{Lem}
\label{lem:summation}
For any integers $k\ge 1$ and $m\ge 2$,
$$
S_k(m)=\sum_{j=0}^k \binom kj B_{k-j}\frac{m^{j+1}}{j+1}.
$$
\end{Lem}

\begin{Lem}[Voronoi congruence]
\label{lem:voronoi}
Let $k$ and $m$ be positive integers, where $m\ge 2$ and $k\ge 2$ is even, then $V_kS_k(m)\equiv U_k m\pmod{m^2}$.
\end{Lem}

The following lemma gives a refinement of the Voronoi congruence (see \cite[Proposition~8.5]{Kellner2007}). 

\begin{Lem}
\label{lem:kellner}
Let $k$ and $m$ be positive integers, where $m\ge 2$ and $k\ge 6$ is even. If $m\mid U_k$ then $V_kS_k(m)\equiv U_k m\pmod{m^3}$.
\end{Lem}

A prime $p$ will be called \emph{regular} if it does not divide any
of the numerators $U_r$ with even $r\le p-3$, otherwise it is said to be \emph{irregular}. The pairs $(r,p)$ with $p\mid U_r$ and even $r\le p-3$ are called \emph{irregular pairs}. 
At a first glance this looks like a strange definition, but by celebrated work
of Kummer (1850) \cite{Kummer} can be reformulated as: a prime $p$ is irregular
if and only if it divides the class number of $\mathbb Q(\zeta_p)$.
The first
few irregular primes are $37, 59, 67, 101, 103, 131, 149,\ldots.$ It is known that
there are infinitely many irregular primes, cf.~Carlitz~\cite{car}. It is not known whether there are infinitely many
regular primes. Let $\pi_{\iota}(x)$ denote
the number of irregular primes $p\le x$.
Recently Luca et al.~\cite[Theorem~1]{LPP} showed that 
$$\pi_{\iota}(x)\ge (1+o(1))\,\frac{\log \log x}{\log \log \log x},\quad x\rightarrow \infty.$$
Conjecturally, cf.~Siegel~\cite{Siegel}, and in good agreement with numerical work, we
should have 
$$\pi_{\iota}(x)\sim \Bigl(1-\frac 1{\sqrt{e}}\Bigr)\pi(x)\approx 0.3935\, \frac x{\log x}.$$
Let $N_{\iota}(x)$ denote the number of integers $n\le x$ that are composed only of irregular primes.
If we assume that 
\begin{equation}
\pi_{\iota}(x)\sim \delta\, \frac x{\log x},\quad 0<\delta<1,
\end{equation}
then by Moree~\cite[Theorem~1]{LvR} we have $N_{\iota}(x)\sim c x\log^{\delta-1}x$ as $x\rightarrow \infty$, with $c$ a positive real constant.
(Kummer conjectured that $\delta=1/2$.) The latter result is of Wirsing type (cf.~Schwarz
and Spilker~\cite[65-76]{SS}).

For more results on Bernoulli numbers see e.g. the book by Arakawa et al.~\cite{AIK}.

\section{The Kellner-Erd\H{o}s-Moser conjecture}
\label{sec:kell}
In this section, we will use properties of Bernoulli numbers to study the non-trivial solutions
of the equation $aS_k(m)=m^k$. This will then lead us to establish Theorem~\ref{thm:t4}. As a bonus we will conclude that
if $aS_k(m)=m^k$ has non-trivial solutions, then $a$ must be either 1 or a product of irregular primes.

First assume that $m=2$. Then $a=2^k$. Next assume that $m=3$. Then we must have $a(1+2^k)=3^k$ and hence
$a=3^e$ for some $e\le k$. It follows that $1+2^k=3^{k-e}$. This Diophantine equation was
already solved by the famous medieval astronomer Levi ben Gerson (1288-1344), alias Leo Hebraeus, who
showed that 8 and 9 are the only consecutive integers in the sequence of powers of 2 and 3, see
Ribenboim~\cite[pp. 124-125]{Ribenboim}. This leads to the solutions $(e,k)\in \{(0,1),(1,3)\}$ and hence
$(a,m,k)\in \{(1,3,1),(3,3,3)\}$. Now assume that $m\ge 4$ and $k$ is odd. Then by Lemma~\ref{lem:staudt} we find that $m(m-1)/2$ divides $m^k$, which is impossible. We infer that, to 
establish Conjecture~\ref{kellie}, it is enough to establish the following conjecture.
\begin{Con}
\label{kellie2}
The set
$$
{\EuScript A}=\{a\ge 1: aS_k(m)=m^k{\rm ~has~a~solution~with~} m\ge 4~{\rm and}~k~{\rm even}\}
$$
is empty.
\end{Con}

\begin{Lem}
\label{lem:l6}
Suppose that $aS_k(m)=m^k$ with $m\ge 4$ and $k$ even. If $p$ is a prime dividing $m$, then $(p-1)\nmid k$.
\end{Lem}

\begin{proof}
Assume that $p\mid m$ and $(p-1)\mid k$. Let $p^e\parallel m$ and $p^f\parallel a$, where $e\ge 1$, $f\ge 0$. Using Lemma~\ref{lem:staudt} we find that 
$S_k(m)\equiv \frac m{p^e}S_k(p^e)\equiv -\frac mp\pmod{p^e}$. Hence $aS_k(m)\equiv-\frac{am}p\pmod{p^{e+f}}$, and so $m^k\equiv-\frac{am}p\pmod{p^{e+f}}$. Since $p^{e+f-1}\parallel\frac{am}p$, we deduce that $p^{e+f-1}\parallel m^k$, and thus $f=(k-1)e+1\ge k$. Then $a\ge p^k\ge 2^k$ and $aS_k(m)>2^k(m-1)^k>m^k$, which contradicts the fact that $aS_k(m)=m^k$.
\end{proof}

\begin{cor}
\label{cor:c1}
Suppose that $aS_k(m)=m^k$ with $m\ge 4$ and $k$ even, then we have $\gcd(m,6)=1$.
\end{cor}

\begin{cor}
\label{cor:c2}
If $\gcd(a,6)\ne 1$, then $a\not\in \EuScript A$.
\end{cor}

\begin{Lem}
\label{lem:l7}
Suppose that $aS_k(m)=m^k$ with $m\ge 4$ and $k$ even. 
Suppose that $a$ has no prime divisor $p$ satisfying $p<2^s$.
Then $a\mid m^{\lceil{k/s}\rceil -1}$.
\end{Lem}
\begin{proof}
Assume that $a\nmid m^{\lceil{k/s}\rceil-1}$. Since each prime divisor of $a$ divides $m$, there exists a prime $p$ such that $p^{\lceil{k/s}\rceil}\mid a$. By assumption $p\ge 2^s$. Then $a\ge (2^s)^{\lceil{k/s}\rceil}\ge 2^k$ and $aS_k(m)\ge 2^k(m-1)^k>m^k$, which is a contradiction.
\end{proof}

On combining the latter lemma and Corollary~\ref{cor:c2} we obtain the following result.
\begin{cor}
\label{cor:s=2}
Suppose that $aS_k(m)=m^k$ with $m\ge 4$ and $k$ even, then
$a\mid m^{(k-2)/2}$.
\end{cor}

\begin{Lem}
\label{lem:l8}
Suppose that $aS_k(m)=m^k$ with $m\ge 4$ and $k$ even, then $m\mid U_k$.
\end{Lem}

\begin{proof}
Multiplying the Voronoi congruence by $a$ and using the fact that $aS_k(m)=m^k$, we 
deduce that $V_km^k\equiv U_k am\pmod{am^2}$. Since $m^2\mid m^{(k+2)/2}$ and, by Corollary~\ref{cor:s=2}, $a\mid m^{(k-2)/2}$, we have $am^2\mid U_kam$, that is $m\mid U_k$.
\end{proof}

\begin{cor}
\label{cor:c3}
Suppose that $aS_k(m)=m^k$ with $m\ge 4$ and $k$ even, then $k\ge 10$.
\end{cor}

\begin{Lem}
\label{lem:l10}
Suppose that $aS_k(m)=m^k$ with $m\ge 4$ and $k$ even, then $m^2\mid U_k$.
\end{Lem}

\begin{proof}
By Corollary~\ref{cor:c3}, $k\ge 10$. Using Lemma~\ref{lem:kellner} instead of the Voronoi congruence and proceeding then by the same argument as in the proof of Lemma~\ref{lem:l8}, we deduce 
that $am^3\mid U_kam$, and so $m^2\mid U_k$.
\end{proof}

Since $U_k$ is square-free for any even $k<50$, we have
\begin{cor}
\label{cor:c4}
Suppose that $aS_k(m)=m^k$ with $m\ge 4$ and $k$ even, then $k\ge 50$.
\end{cor}

In case $a=1$ the next result is Lemma~10 of \cite{MRU}.
\begin{Thm}
\label{thm:t4}
Suppose that $aS_k(m)=m^k$ with $m\ge 4$ and even $k$. Let $p$ be a prime dividing $m$. Then 

\noindent\textup{(a)} $p$ is an irregular prime;

\noindent\textup{(b)} $k\not\equiv 0,2,4,6,8,10,14\pmod{p-1}$;

\noindent\textup{(c)} $\ord_p (B_k/k)\ge 2\ord_p m\ge 2$;

\noindent\textup{(d)} $k\equiv r\pmod{p-1}$ for some irregular pair $(r,p)$. 
\end{Thm}

\begin{proof}
By Corollary~\ref{cor:c1} and Lemmas~\ref{lem:l6} and \ref{lem:l8} we see that $p\ge 5$, $(p-1)\nmid k$ and $p\mid U_k$. If $p\nmid k$, then $\ord_p(B_k/k)>0$ and $p$ is irregular. Now assume that $p\mid k$, i.e., 
$\ord_p k\ge 1$. In view of Corollary~\ref{cor:c3} we 
have $k\ge 10$. Further
$$
S_k(m)=B_km+\frac{k(k-1)}6 B_{k-2}m^3+\sum_{j=4}^k \binom kj B_{k-j}\frac{m^{j+1}}{j+1}.
$$
Hence
\begin{equation}
\label{bernoulli}
\frac{m^{k-1}}{ak}=\frac{B_k}k+\frac{k-1}6 B_{k-2}m^2+m^2\sum_{j=4}^k \frac{(k-1)!}{(k-j)!}B_{k-j}\frac{m^{j-2}}{(j+1)!}.
\end{equation}
By Lemma~\ref{lem:clausen} we have $\ord_p V_{k-2}\le 1$, and hence
\begin{equation}
\label{order1}
\ord_p\left(\frac{k-1}6 B_{k-2}m^2\right)\ge 2\ord_p m-\ord_p V_{k-2}\ge 1.
\end{equation}
Further, for $j=4,5,\dots,k$,
$$
\ord_p\left(\frac{m^{j-2}}{(j+1)!}\right)=(j-2)\ord_p m-\frac{j+1-\sigma_p(j+1)}{p-1}\ge j-2-\frac j4\ge 1,
$$
where $\sigma_p(j+1)$ denotes the sum of the digits of $j+1$ written in the base $p$. Therefore
\begin{equation}
\label{order2}
\ord_p\left(m^2\sum_{j=4}^k\frac{(k-1)!}{(k-j)!}B_{k-j}\frac{m^{j-2}}{(j+1)!}\right)\ge 2\ord_p m\ge 2.
\end{equation}
Note that
\begin{align}
\ord_p\left(\frac{m^{k/2}}k\right)&=\frac k2\ord_p m-\ord_p k 
\ge\frac 12\, p^{\ord_p k}-\ord_p k\notag\\
&>2^{2\ord_p k-1}-\ord_p k\ge\ord_p k\ge 1.\label{order3}
\end{align}
Using Corollary~\ref{cor:s=2} we obtain
\begin{equation}
\label{orderr2}
\ord_p\left(\frac{m^{(k-2)/2}}a\right)\ge 0.
\end{equation}
It follows from \eqref{order3} and \eqref{orderr2} that
\begin{equation}
\label{orderr3}
\ord_p\left(\frac{m^{k-1}}{ak}\right)=\ord_p\left(\frac{m^{(k-2)/2}}a\cdot
\frac{m^{k/2}}k\right)\ge 1.
\end{equation}
Combining \eqref{bernoulli}~--~\eqref{order2} and 
\eqref{orderr3}, we deduce that $\ord_p(B_k/k)>0$, and so $p$ is irregular. This completes the proof of part~(a). Part~(b) is a consequence of part~(a), Lemma~\ref{lem:l6} and the Kummer congruence. In the case $p\nmid k$, part~(c) follows from Lemma~\ref{lem:l10}. Now assume that $p\mid k$. By part~(a), $p$ is an irregular prime, and so $p\ge 37>2^5$. 
On combining Lemma~\ref{lem:l7} with $s=5$ and Corollary~\ref{cor:c4}, \eqref{orderr2} is sharpened to
$$\ord_p\left(\frac{m^{(k-2)/2}}a\right)>2\ord_p m.$$
Combining the latter estimate with \eqref{order3} gives
\begin{align}
\ord_p\left(\frac{m^{k-1}}{ak}\right)>2\ord_p m.\label{order4}
\end{align}
Further, by part~(b) and the von Staudt-Clausen theorem, $\ord_p V_{k-2}=0$. Combining \eqref{bernoulli}, \eqref{order1}, \eqref{order2} and \eqref{order4}, we complete the proof of part~(c). Part~(d) is a direct consequence of part~(c), the fact that $(p-1)\nmid k$, and the Kummer congruence.
\end{proof}

\begin{cor}
\label{cor:c5}
If $a$ has a regular prime divisor, then $a\not\in\EuScript{A}$.
\end{cor}

\begin{cor}
\label{cor:c6}
Let $p_1$ and $p_2$ be distinct irregular prime divisors of $a$. Assume that for every pair $(r_1,p_1)$, $(r_2,p_2)$ of irregular pairs we have $\gcd(p_1-1,p_2-1)\nmid(r_1-r_2)$. Then $a\not\in\EuScript{A}$.
\end{cor}
\noindent {\tt Example}. Suppose that $37\cdot 379\mid a$. If $a\in\EuScript A$ then 
$aS_k(m)=m^k$  with $m\ge 4$ and $k$ even and both $37$ and $379$ must divide $m$. There is one irregular pair $(32,37)$ corresponding to $37$ and two irregular pairs $(100,379)$ and $(174,379)$ corresponding to $379$. By 
Theorem~\ref{thm:t4} (d), $k$ must be a simultaneous solution of the congruences $k\equiv 32\pmod{36}$ and $k\equiv 100\;\text{or\;}174\pmod{378}$, which is impossible as $\gcd(36,378)=18$, $18\nmid (32-100)$ and $18\nmid(32-174)$. Hence $a\not\in\EuScript A$.

\section{Helpful pairs}
\label{sec:hp}
Helpful pairs will be used to show that certain ratios are forbidden (Section \ref{exclusive}) and to show that certain 
numbers have to divide $k$ (Section~\ref{sec:cascade}). In both  cases one has to exclude that $k$ is in certain congruence
classes. In order to show that a certain ratio is forbidden, we have to exclude {\it all} the congruence 
classes with an appropriate modulus.
In order to show that a certain even number $d$ divides $k$, we do this by excluding all the
congruence classes $2i\pmod{d}$ for $1\le i<d/2$. 
If $p\mid{a}$ is an irregular prime, $d=p-1$, then by Theorem~\ref{thm:t4} we immediately exclude many
congruence classes.\\
\indent The exclusion of a congruence is achieved by a helpful pair and
the procedure is described just after the proof of the crucial Lemma~\ref{helpful}.
\begin{Def}
For a positive integer $a$ let us call a pair $(t,q)_a$ with $q$ a prime and $2\le t\le q-3$ even to be \emph{helpful} if $q\nmid a$
and
$aS_t(c)\not\equiv c^t\pmod{q}$ for every integer 
$c$ satisfying $1\le c\le q-1$. 
If $q$ is an irregular prime, we require in
addition that $(t,q)$ should not be an irregular pair.
\end{Def}

\begin{Prop}
Let $q\ge 5$ be a prime and $a$ be a positive integer. Then $(2,q)_a$ is a helpful pair if and only if $\bigl(\frac{a^2+36a+36}q\bigr)=-1$.
\end{Prop}

\begin{proof}
Note that $B_2=1/6$ and hence $(2,q)$ cannot be an irregular pair. Since $S_2(c)=(2c^3-3c^2+c)/6$, we see that $(2,q)_a$ is a helpful pair if and only if $q\nmid a$ and $a(2c^3-3c^2+c)\not\equiv 6c^2\pmod{q}$ for $c=1,\dots,q-1$, that is, 
if and only if $a(2c^2-3c+1)\not \equiv 6c\pmod{q}$ for $c=0,\ldots,q-1$, i.e., if and only if we have
$\bigl(\frac{9(a+2)^2-8a^2}q\bigr)=\bigl(\frac{a^2+36a+36}q\bigr)=-1$.
\end{proof}

\begin{Prop}
Let $q\ge 7$ be a prime with $\bigl(\frac{31}q\bigr)=-1$. Then $(4,q)_{q-2}$ is a helpful pair.
\end{Prop}

\begin{proof}
From $\bigl(\frac{31}q\bigr)=-1$ we deduce that $6n^2+10n-1\not\equiv 0\pmod{q}$ for $n=1,\dots,q-1$. This implies that $30(q-2)S_4(c)-30c^4\equiv -2c(6c^4+10c^2-1)\not\equiv 0\pmod{q}$ for $c=1,\dots,q-1$. Since $B_4=-1/30$, $(4,q)$ is not an irregular pair. Thus $(4,q)_{q-2}$ is a helpful pair. 
\end{proof}

\begin{Lem}
\label{helpful}
Let $2\le t\le q-3$ and $q$ be a prime.
If $(t,q)_a$ is a helpful pair and $aS_k(m)=m^k$ with $k$ even, then 
we have $k\not\equiv t\pmod{q-1}$.
\end{Lem}

\begin{proof}
Assume that $k\equiv t\pmod{q-1}$. By Theorem~\ref{thm:t4}~(d) we must have $q\nmid m$, for otherwise $(t,q)$ is 
an irregular pair, contradicting the definition of a helpful pair. Thus we can write 
$m=m_0q+b$ with $1\le b\le q-1$. By Lemma~\ref{lem:staudt} we have $q\mid S_t(q)$.
We now find, modulo $q$, $S_k(m)\equiv S_t(m)\equiv m_0S_t(q)+S_t(b)\equiv S_t(b)$. Thus if \eqref{KEME} is satisfied we must have $aS_t(b)\equiv b^t\pmod{q}$. By the definition of a helpful pair this
is impossible.
\end{proof}

\noindent {\it Ruling out congruence classes for $k$}. The helpful pairs give us a chance to rule out $k$ that satisfy certain congruences of the
form $k\equiv c\pmod{d}$ with $c\ge 2$ and $d$ even. We first list all primes $q\ge 5$ such
that $q-1$ divides $d$. Denote these primes by $q_1,\ldots,q_s$. Let $t_i$ denote the least nonnegative integer congruent to $c$ modulo $q_i-1$. If one of the pairs
$(t_i,q_i)_a$ is helpful, by Lemma~\ref{helpful} we have ruled out $k\equiv c\pmod{d}$.
If this does not work, we multiply $d$ by an integer $\ell\ge 2$ (\emph{the lifting factor}). Our original congruence is now replaced by $\ell$ congruences, 
$k\equiv c+jd\pmod{\ell d}$, $0\le j<\ell$. For each of these congruences we now continue as before.
In certain cases we find that each of the lifted congruences is ruled out by a helpful pair and then
we are done. This situation is described in Proposition~\ref{prop:t5} below. If not all of the lifted
congruences are excluded by helpful pairs, we can lift the bad congruences still further.
The above procedure is not systematic and each stage the danger lurks that we get too many congruence 
classes we cannot exclude anymore.

\begin{Prop}
\label{prop:t5}
Let $p$ be an irregular prime dividing $a$. Assume that for every irregular pair $(r,p)$ there exists a positive integer $\ell_r$ such that for every $j=0,1,\dots,\ell_r-1$ there is a helpful pair $(t_j,q_j)_a$ with $(q_j-1)\mid \ell_r(p-1)$ and $t_j\equiv r+j(p-1)\pmod{q_j-1}$. Then $a\not\in\EuScript{A}$.
\end{Prop}

\begin{proof}
Since $p$ must divide $m$, Theorem~\ref{thm:t4}~(d) yields $k\equiv r\pmod{p-1}$ for some irregular pair $(r,p)$. Hence there exists $j\in\{0,1,\dots,\ell_r-1\}$ such that $k\equiv r+j(p-1)\pmod{\ell_r(p-1)}$. Then we have $k\equiv t_j\pmod{q_j-1}$ for the helpful pair $(t_j,q_j)_a$, which contradicts Lemma~\ref{helpful}.
\end{proof}

\begin{cor}
Under the conditions of Proposition~$\ref{prop:t5}$, we have $ab\not\in\EuScript{A}$ for any positive integer $b\equiv 1\pmod{Q}$, where $Q$ denotes the least common multiple of all components $q_j$ of helpful pairs constructed for all irregular pairs $(r,p)$ corresponding to~$p$.
\end{cor}

\section{Excluding a given ratio $\rho$}
\label{exclusive}
Let $\rho\ge 3$. Write $a=\rho-1$. If $a$ has no regular prime divisor, the only
way we know to exclude $\rho$ is by using helpful pairs or invoking Corollary~\ref{cor:c6}. We demonstrate this
with two examples (an example of the usage of Corollary \ref{cor:c6} we already gave immediately following the 
statement of Corollary~\ref{cor:c6}).
\smallskip

\noindent {\tt Easy example}: $a=673$.\\ 
\noindent There are two irregular pairs $(408,673)$ and $(502,673)$ corresponding to $673$. Theorem~\ref{thm:t4}~(d) yields $k\equiv 408\,\,\text{or}\,\,502\pmod{672}$. If $k\equiv 408\pmod{672}$, then $k\equiv 8\pmod{16}$. The latter is impossible since 
$(8,17)_{10}$ is a helpful pair by Table~\ref{table2}.
If $k\equiv 502\pmod{672}$, then $k\equiv 2\pmod{4}$, which is impossible as $(2,5)_3$ is a helpful pair by Table~\ref{table2}.
\smallskip

\noindent {\tt Difficult example}: $a=653$.\\ 
\noindent There is one irregular pair $(48,653)$ corresponding to 
$653$, and so $k\equiv 48\pmod{652}$. We have $652=2^2\cdot 163$. There are no helpful
pairs $(48,q)$ with $(q-1)\mid{652}$. So we have to use a lifting factor $\ell$. It turns out
that $\ell=4$ is a useful factor. So that is why we use it in the first step.

\emph{Step $1$.} We have $k\equiv 48\,\,\text{or}\,\,700\,\,\text{or}\,\,1352\,\,\text{or}\,\,2004\pmod{2608}$. The case $k\equiv 48\pmod{2608}$ is impossible, since $(48,2609)_{653}$ is a helpful pair. If $k\equiv 700\,\,\text{or}\,\,1352\pmod{2608}$, then $k\equiv 8\,\,\text{or}\,\,12\pmod{16}$, which is impossible as $(8,17)_7$ and $(12,17)_7$ are helpful pairs. Thus $k\equiv 2004\pmod{2608}$. 

\emph{Step $2$.} We have $k\equiv 2004\,\,\text{or}\,\,4612\,\,\text{or}\,\,7220\,\,\text{or}\,\,9828\,\,\text{or}\,\,12436\pmod{13040}$. If $k\equiv 4612\,\,\text{or}\,\,7220\,\,\text{or}\,\,9828\pmod{13040}$, then $k\equiv 12\,\,\text{or}\,\,20\,\,\text{or}\,\,28\pmod{40}$. From the fact that $(12,41)_{38}$, $(20,41)_{38}$ and $(28,41)_{38}$ are helpful pairs we deduce that the latter congruence is impossible. Hence $k\equiv 2004\,\,\text{or}\,\,12436\pmod{13040}$. 

\emph{Step $3$.} We have $k\equiv 2004\,\,\text{or}\,\,12436\,\,\text{or}\,\,15044\,\,\text{or}\,\,25476\,\,\text{or}\,\,28084\,\,\text{or}\,\,38516$\linebreak $\pmod{39120}$. If $k\equiv 15044\,\,\text{or}\,\,28084\pmod{39120}$ then $k\equiv 4\,\,\text{or}\,\,14\pmod{30}$, which is impossible as $(4,31)_2$ and $(14,31)_2$ are helpful pairs. If $k\equiv 2004\,\,\text{or}\,\,25476\pmod{39120}$, then $k\equiv 24\,\,\text{or}\,\,36\pmod{60}$. The latter is impossible, since $(24,61)_{43}$ and $(36,61)_{43}$ are helpful pairs. The case $k\equiv 12436\pmod{39120}$ implies $k\equiv 196\pmod{240}$, which is impossible as $(196,241)_{171}$ is a helpful pair. The remaining case $k\equiv 38516\pmod{39120}$ is also impossible, since in this case $k\equiv 9176\pmod{9780}$ and $(9176,9781)_{653}$ is a helpful pair.\\

\noindent {\tt Remark}. Using helpful pairs, we can find some infinite families of forbidden ratios. For example, let $\rho=37^{s}+1$ for some positive integer $s$. 
The prime $37$ is irregular and $(32,37)$ an irregular pair.
If $37^s S_k(m)=m^k$ with $k$ even then, by Theorem~\ref{thm:t4}~(d), we have $k\equiv 32\pmod{36}$. This implies that $k\equiv 8\pmod{12}$. Since $(8,13)_{37^s}$ is a helpful pair if and only if $37^s\equiv 1\,\,\text{or}\,\,2\,\,\text{or}\,\,6\,\,\text{or}\,\,8\,\,\text{or}\,\,11\pmod{13}$ (see Table~\ref{table2}), we deduce that $\rho=37^{s}+1$ is a forbidden ratio for any $s\equiv 0\,\,\text{or}\,\,1\,\,\text{or}\,\,7\,\,\text{or}\,\,9\,\,\text{or}\,\,11\pmod{12}$.

\section{Bad ratios}
\label{sec:bad}
Table~\ref{table3} gives a list of ratios we excluded and the helpful pairs used to do so.
The attentive reader will notice that various ratios $\rho$ are apparently bad and difficult
to exclude. These are related to $a=\rho-1$ that are of the 
form $(2p+1)^s$ with $p$ a \emph{Sophie Germain prime}. 
Recall that a prime $p$ is said to be a \emph{Sophie Germain prime} if also $2p+1$ is a prime. Heuristics suggests that there ought to be infinitely many Sophie Germain primes such that $2p+1$ is an irregular prime.
\begin{Con}
There are infinitely many primes $p$ such that $2p+1$ is an irregular prime.
\end{Con}
Let $p$ be a prime such that $2p+1$ is an irregular prime and let $(r,2p+1)$ be an irregular pair. 
In case we want to rule out $k\equiv r\pmod{2p}$ we are in bad shape to start with. We are directly
forced here to use a lifting factor $\ell\ge 2$ (as the list of primes $5\le q<2p+1$ with $(q-1)\mid2p$ is empty here).
The next result shows that we are in even worse shape, since helpful pairs with primes $q=2pu+1>6p$ have
to be used.
\begin{Prop}
\label{prop:t10}
Let $p$ be a prime such that $2p+1$ is an irregular prime dividing $a$ and let $(r,2p+1)$ be an irregular pair. Let $\ell$ be a positive integer with $p\nmid\ell$, let $q_0,q_1,\dots,q_{\ell-1}$ be odd primes with $(q_j-1)\mid 2\ell$ (not necessarily distinct) and let $t_0,t_1,\dots,t_{\ell-1}$ be positive integers satisfying the conditions $t_j\equiv r+2pj\pmod{q_j-1}$, $0\le j\le\ell-1$. Then at least one of the pairs $(t_0,q_0)_a,(t_1,q_1)_a,\dots,(t_{\ell-1},q_{\ell-1})_a$ is not a helpful pair. 
\end{Prop}

\begin{proof}
Since $r$ is even and $p\nmid\ell$, there exists a $j$ with $0\le j\le\ell-1$ such that $pj\equiv -r/2\pmod{\ell}$. Hence $2\ell\mid(r+2pj)$. This implies that $(q_j-1)\mid t_j$, and so $(t_j,q_j)_a$ is not a helpful pair.
\end{proof}

In case we are not able to exclude such a bad ratio, we might try at least to show that
the $k$ of a solution has to be highly divisible. In the next section we demonstrate this for the
bad ratio $6780$.

\section{Divisibility of $k$}
\label{sec:cascade}
In this section, we consider the case $a=6779=2\cdot 3389+1$ and show that for a non-trivial solution $k$ is divisible
by a large number.
We will present an heuristic argument here why we think that for this $a$ there are no solutions. We expect that a similar reasoning might work for other values of $a$ as well, once one
can establish that a smallish number like $120$ divides $k$.

\indent We start by discussing a baby example.
\begin{Prop}$~~$\\
\noindent\textup{(a)} If $a\equiv 1\,\,\text{or}\,\,2\,\,\text{or}\,\,3\pmod{5}$, then $4\mid{k}$.\\
\noindent\textup{(b)} If $a\equiv 1\,\,\text{or}\,\,3\,\,\text{or}\,\,5\pmod{7}$, then $6\mid{k}$.\\
\noindent\textup{(c)} If $a\equiv 6\,\,\text{or}\,\,7\pmod{11}$, then $10\mid{k}$.\\
\noindent\textup{(d)} If $a\equiv 2\,\,\text{or}\,\,8\,\,\text{or}\,\,11\pmod{13}$, then $12\mid{k}$.\\
\noindent\textup{(e)} If $a\equiv 1\,\,\text{or}\,\,6\pmod{13}$, then $6\mid k$.\\
\noindent\textup{(f)} If $a\equiv 1\,\,\text{or}\,\,5\pmod{11}$ and $a\equiv 15\pmod{31}$, then $10\mid{k}$.
\end{Prop}
\begin{proof}
If, e.g., $a\equiv 2\pmod{13}$, we see from Table~\ref{table2} that the 
pairs $(2,13)_a$, $(4,13)_a$, $(6,13)_a$, $(8,13)_a$ and $(10,13)_a$ are all helpful.
The final assertion follows from glancing at an extended version of Table~\ref{table2}.
\end{proof}

Now let us consider a more serious example, with $\rho=6780$ a bad ratio.
\begin{Prop}
\label{6779}
If\, $6779S_k(m)=m^k$, then $2^6\cdot 3^3\cdot 5^2\cdot7\cdot11\cdot 13\mid k$.
\end{Prop}
\begin{proof}
We start with the congruence $k\equiv 3994\pmod{6778}$, which is a consequence of Theorem~\ref{thm:t4}~(d) and the fact that there is only one irregular pair $(3994,6779)$ corresponding to $6779$. 

\emph{Step $1$.} We have $k\equiv 3994\,\,\text{or}\,\,10772\,\,\text{or}\,\,17550\pmod{20334}$. If 
$k\equiv 3994\,\,\text{or}\,\,10772\pmod{20334}$ then $k\equiv 2\,\,\text{or}\,\,4\pmod{6}$, which is impossible as $(2,7)_3$ and $(4,7)_3$ are helpful pairs. Hence $k\equiv 17550\pmod{20334}$ and $2\cdot 3\mid k$.

\emph{Step $2$.} We have $k\equiv 17550\,\,\text{or}\,\,37884\,\,\text{or}\,\,58218\pmod{61002}$. The case $k\equiv 58218\pmod{61002}$ is impossible, since in this case $k\equiv 6\pmod{18}$ and $(6,19)_{15}$ is a helpful pair. Hence $k\equiv 17550\,\,\text{or}\,\,37884\pmod{61002}$.

\emph{Step $3$.} We have $k\equiv 17550\,\,\text{or}\,\,37884\,\,\text{or}\,\,78552\,\,\text{or}\,\,98886\pmod{122004}$. If $k\equiv 17550\,\,\text{or}\,\,98886\pmod{122004}$, then $k\equiv 18\,\,\text{or}\,\,30\pmod{36}$, which is impossible as $(18,37)_8$ and $(30,37)_8$ are helpful pairs. Hence $k\equiv 37884\,\,\text{or}\,\,78552\pmod{122004}$.

\emph{Step $4$.} We have $k\equiv 37884\,\,\text{or}\,\,78552\,\,\text{or}\,\,159888\,\,\text{or}\,\,200556\pmod{244008}$. If $k\equiv 37884\,\,\text{or}\,\,159888\,\,\text{or}\,\,200556\pmod{244008}$, then $k\equiv 12\,\,\text{or}\,\,36\,\,\text{or}\,\,48\pmod{72}$, which is impossible as $(12,73)_{63}$, $(36,73)_{63}$ and $(48,73)_{63}$ are helpful pairs. Hence $k\equiv 78552\pmod{244008}$ and $2^3\cdot 3^2\mid k$.

\emph{Step $5$.} We have $k\equiv 78552\,\,\text{or}\,\,322560\pmod{488016}$. In the case $k\equiv 78552\pmod{488016}$ we have $k\equiv 8\pmod{16}$, which is impossible as $(8,17)_{13}$ is a helpful pair. Hence $k\equiv 322560\pmod{488016}$ and $2^4\cdot 3^2\mid k$.

\emph{Step $6$.} We have $k\equiv 322560\,\,\text{or}\,\,810576\pmod{976032}$. The case $k\equiv 810576\pmod{976032}$ is impossible, since in this case $k\equiv 48\pmod{96}$ and $(48,97)_{86}$ is a helpful pair. Hence $k\equiv 322560\pmod{976032}$ and $2^5\cdot 3^2\mid k$.

\emph{Step $7$.} We have $k\equiv 322560\,\,\text{or}\,\,1298592\pmod{1952064}$. In the case $k\equiv 1298592\pmod{1952064}$ we have $k\equiv 288\pmod{576}$. The latter is impossible, since $(288,577)_{432}$ is a helpful pair. Hence $k\equiv 322560\pmod{1952064}$ and $2^6\cdot 3^2\mid k$.

\emph{Step $8$.} We have $k\equiv 322560\,\,\text{or}\,\,2274624\,\,\text{or}\,\,4226688\pmod{5856192}$. The case $k\equiv 322560\pmod{5856192}$ is impossible, since in this case $k\equiv 288\pmod{432}$ and $(288,433)_{284}$ is a helpful pair. In the case $k\equiv 2274624\pmod{5856192}$ we have $k\equiv 36\pmod{108}$, which is impossible as $(36,109)_{21}$ is a helpful pair. Hence $k\equiv 4226688\pmod{5856192}$ and $2^6\cdot 3^3\mid k$.

\emph{Step $9$.} We have $k\equiv 4226688\,\,\text{or}\,\,10082880\,\,\text{or}\,\,15939072\,\,\text{or}\,\,21795264\,\,\text{or}\,\,27651456\pmod{29280960}$. If $k\equiv 15939072\,\,\text{or}\,\,27651456\pmod{29280960}$, then $k\equiv 6\,\,\text{or}\,\,12\pmod{30}$, which is impossible as $(6,31)_{21}$ and $(12,31)_{21}$ are helpful pairs. In the case $k\equiv 21795264\pmod{29280960}$ we have $k\equiv 24\pmod{60}$, which is impossible since $(24,61)_8$ is a helpful pair. The case $k\equiv 4226688\pmod{29280960}$ is also impossible, since in this case $k\equiv 108\pmod{180}$ and $(108,181)_{82}$ is a helpful pair. Hence $k\equiv 10082880\pmod{29280960}$ and $2^6\cdot 3^3\cdot 5\mid k$.

\emph{Step $10$.} We have $k\equiv 10082880+29280960j\pmod{204966720}$ for some $j\in\{0,1,\dots,6\}$. If $j\in\{0,1,3,6\}$ then $k\equiv 8\,\,\text{or}\,\,12\,\,\text{or}\,\,16\,\,\text{or}\,\,24\pmod{28}$, which is impossible as $(8,29)_{22}$, $(12,29)_{22}$, $(16,29)_{22}$ and $(24,29)_{22}$ are helpful pairs. If $j=4$ or 5, then we have $k\equiv 6\,\,\text{or}\,\,18\pmod{42}$. The latter is impossible, since $(6,43)_{28}$ and $(18,43)_{28}$ are helpful pairs. Hence $k\equiv 68644800\pmod{204966720}$ and ${2^6\cdot 3^3\cdot 5\cdot 7\mid k}$.

\emph{Step $11$.} We have $k\equiv 68644800+204966720j\pmod{1024833600}$ for some $j\in\{0,1,2,3,4\}$. If $j\in\{1,2,3\}$, then $k\equiv 20\,\,\text{or}\,\,40\,\,\text{or}\,\,60\pmod{100}$, which is impossible as $(20,101)_{12}$, $(40,101)_{12}$ and $(60,101)_{12}$ are helpful pairs. The case $j=4$ is also impossible, since in this case $k\equiv 3780\pmod{6300}$ and $(3780,6301)_{478}$ is a helpful pair. Hence $k\equiv 68644800\pmod{1024833600}$ and $2^6\cdot 3^3\cdot 5^2\cdot 7\mid k$.

\emph{Step $12$.} We have $k\equiv 68644800+1024833600j\pmod{11273169600}$ for some $j\in\{0,1,\dots,10\}$. If $j\in\{0,2,3,6,8\}$, then $k\equiv 2\,\,\text{or}\,\,4\,\,\text{or}\,\,12\,\,\text{or}\,\,18\,\,\text{or}\,\,20\pmod{22}$, which is impossible as $(2,23)_{17}$, $(4,23)_{17}$, $(12,23)_{17}$, $(18,23)_{17}$ and $(20,23)_{17}$ are helpful pairs. If $j\in\{1,5,7\}$, then $k\equiv 30\,\,\text{or}\,\,36\,\,\text{or}\,\,54\pmod{66}$, which is impossible since $(30,67)_{12}$, $(36,67)_{12}$ and $(54,67)_{12}$ are helpful pairs. In the case $j=10$ we have $k\equiv 16\pmod{88}$, which is impossible as $(16,89)_{15}$ is a helpful pair. The case $j=9$ is also impossible, since in this case $k\equiv 160\pmod{352}$ and $(160,353)_{72}$ is a helpful pair. Hence $k\equiv 4167979200\pmod{11273169600}$ and $2^6\cdot 3^3\cdot 5^2\cdot 7\cdot 11\mid k$.

\emph{Step $13$.} We have $k\equiv 4167979200+11273169600j\pmod{146551204800}$ for some $j\in\{0,1,\dots,12\}$. If $j\in\{1,3,4,5,7,8,10\}$, then $k\equiv 12\,\,\text{or}\,\,20\,\,\text{or}\,\,24\,\,\text{or}\,\,28\,\,\text{or}\,\,36\,\,\text{or}$ $40\,\,\text{or}\,\,48\pmod{52}$, which is impossible as $(12,53)_{48}$, $(20,53)_{48}$, $(24,53)_{48}$, $(28,53)_{48}$, $(36,53)_{48}$, $(40,53)_{48}$ and $(48,53)_{48}$ are helpful pairs. If $j=2$ or $12$, then $k\equiv 30\,\,\text{or}\,\,42\pmod{78}$, which is impossible, since $(30,79)_{64}$ and $(42,79)_{64}$ are helpful pairs. If $j=0$ or $6$, then $k\equiv 60\,\,\text{or}\,\,110\pmod{130}$, which is impossible as $(60,131)_{98}$ and $(110,131)_{98}$ are helpful pairs. The case $j=9$ is also impossible, since in this case $k\equiv 96\pmod{156}$ and $(96,157)_{28}$ is a helpful pair. Hence $k\equiv 128172844800\pmod{146551204800}$ and $2^6\cdot 3^3\cdot 5^2\cdot 7\cdot 11\cdot 13\mid k$.
\end{proof}

It seems that the type of argument 
used in the proof of Proposition~\ref{6779}    
can be continued to deduce that more and more small prime
factors must divide $k$. Given a prime $q\ge 5$ and $2\le t\le q-3$ even, one would heuristically expect 
that $(t,q)_a$ is helpful with probability $(1-1/q)^{q-1}$ which tends to $1/e$, on assuming that the values
$S_t(c)$ are randomly distributed modulo $q$. The numerical data obtained so far turn out to be consistent with this.

For the original Erd\H{o}s-Moser equation it is known (cf.~\cite{Kellner1, MRU}) that $N\mid k$ with

$$N=2^8\cdot 3^5\cdot 5^4\cdot 7^3\cdot 11^2\cdot 13^2\cdot 17^2\cdot 19^2\cdot 23\cdots
997>5.7462\cdot 10^{427}.$$
An heuristic argument can be given suggesting that if, say $L_v:={\rm lcm}(1,2,\ldots,v)$ divides $k$, with tremendously
high likelihood we can infer that $L_w$ divides $k$, where $w$ is the smallest prime  not dividing $L_v$. 
It is already enough to have $v\ge 11$ here.
To deduce that $k$ is divisible by say $24$ 
might be delicate, but once one has $L_v\mid k$ say, there is an explosion of further 
helpful pairs one can use to establish divisibility of $k$ by an even larger integer.
To add the first prime $w$ not dividing $L_v$, one needs to have only a number of helpful pairs
that is roughly linear in $v$, whereas an exponential number (in $v$) is available.
However, the required computation time goes sharply up with increasing $w$.

This result gives a lower bound of $10^{427}$ for $k$, which is modest in comparison with the lower bound
obtained by Moser. However, as argued by Gallot et al.~\cite{GMZ}, a result of the form $N\mid{k}$ leads to
an expected lower bound $m>10^{257N}$. For the Kellner-Erd\H{o}s-Moser equation we likewise expect a
result of the form $N\mid{k}$ to lead to a lower bound for $m$ that is exponential in $N$.

Unfortunately, the authors are not aware of any systematic approach that would
allow one to prove a result of the type that if $aS_k(m)=m^k$, then $120\mid k$, for every $a\ge 1$. 
Some preliminary work on this 
for the equation $S_k(m)=am^k$ was done by the second author's intern Muriel Lang \cite{Muriel} in 2009.

\section{Lower bound for $m$}
\label{sec:lower}
The aim of this section is to establish Theorem~\ref{thm:t6}. The proof rests on Lemmas~\ref{lem:l11} and \ref{lem:l12}.
\begin{Lem}
\label{lem:l11}
Suppose that $aS_k(m)=m^k$ with $m\ge 4$ and $k$ even. Then $m-1$ and $2m-1$ are square-free, and if $p$ is a prime divisor of 
$(m-1)(2m-1)$, then $(p-1)\mid k$.
\end{Lem}

\begin{proof}
Since
$$
S_k(m-1)=S_k(m)-(m-1)^k\equiv S_k(m)\pmod{m-1},
$$
Lemma~\ref{lem:staudt} yields
\begin{equation}
\label{eq10}
a\sum_{\substack{p\mid(m-1)\\ (p-1)\mid k}}\frac{m-1}p +m^k\equiv 0\pmod{a(m-1)}.
\end{equation}
Note that if $p\mid(m-1)$ and $(p-1)\nmid k$, then $p\mid m$, a contradiction that shows
that $p\mid(m-1)$ implies $(p-1)\nmid k$.
If $p^2\mid(m-1)$ it follows again that $p\mid m$, a contradiction that shows that
$m-1$ is square-free.\\ 
Note that
$$
S_k(2m-1)=\sum_{j=1}^{m-1}(j^k+(2m-1-j)^k)\equiv 2S_k(m)\pmod{2m-1}.
$$
Then, again by Lemma~\ref{lem:staudt},
\begin{equation}
\label{eq11}
a\sum_{\substack{p\mid(2m-1)\\(p-1)\mid k}}\frac{2m-1}p+2m^k\equiv 0\pmod{a(2m-1)},
\end{equation}
from which we deduce that $2m-1$ is square-free and each prime $p$ dividing $2m-1$ satisfies $(p-1)\mid k$.
\end{proof}

\begin{cor}
\label{cor:c3mod4}
Suppose that $aS_k(m)=m^k$ with $m\ge 4$ and $k$ even, then $m\equiv 3\pmod{4}$.
\end{cor}

\begin{Lem}
\label{lem:l12}
Suppose that $aS_k(m)=m^k$ with $m\ge 4$ and $k$ even and let $p$ be a prime divisor of $(m+1)(2m+1)$. If $(p-1)\mid k$ then $\ord_p((m+1)(2m+1))=\ord_p(a+1)+1$, otherwise $\ord_p((m+1)(2m+1))\le\ord_p(a+1)$.
\end{Lem}

\begin{proof}
Observe that
$aS_k(m+1)=(a+1)m^k$. Invoking Lemma~\ref{lem:staudt}, we obtain
\begin{equation}
\label{eq12}
a\sum_{\substack{p\mid(m+1)\\(p-1)\mid k}}\frac{m+1}p+(a+1)m^k\equiv 0\pmod{a(m+1)}.
\end{equation}
Since $p\mid a$ implies $p\mid m$ it follows that $\gcd(a,m+1)=1$. Thus
\begin{align*}
&\ord_p(m+1)=\ord_p(a+1)+1&\text{if $p\mid(m+1)$ and $(p-1)\mid k$,}\\
&\ord_p(m+1)\le\ord_p(a+1)&\text{if $p\mid(m+1)$ and $(p-1)\nmid k$.}
\end{align*}
Further, from
$$
aS_k(2m+1)=a\sum_{j=1}^m (j^k+(2m+1-j)^k)\equiv 2aS_k(m+1)\equiv 2(a+1)m^k\pmod{a(2m+1)}
$$
we deduce that
\begin{equation}
\label{eq13}
a\sum_{\substack{p\mid(2m+1)\\(p-1)\mid k}}\frac{2m+1}p+2(a+1)m^k\!\!\equiv 0\pmod{a(2m+1)},
\end{equation}
and so
\begin{align*}
&\ord_p(2m+1)=\ord_p(a+1)+1&\text{if $p\mid(2m+1)$ and $(p-1)\mid k$,}\\
&\ord_p(2m+1)\le\ord_p(a+1)&\text{if $p\mid(2m+1)$ and $(p-1)\nmid k$.}
\end{align*}
Since $m+1$ and $2m+1$ are coprime, the asserted result follows.
\end{proof}
\noindent Part (g) below arose in collaboration with Jan B\"uthe (University of Bonn) and we only
provide a sketch of the proof here. In a planned sequel to this paper \cite{BBM} further
details will be given. We remark that if the condition $m\equiv 1\pmod{30}$ is replaced
by 
$$\sum_{p|(m-1)}\frac{1}{p}+\frac{1}{a}>1,$$
(cf.~equation \eqref{laatsteloodjes}) the same conclusion holds true.
\begin{Thm}
\label{thm:t6}
Assume that $a>1$ is square-free and that $aS_k(m)=m^k$ with $m\ge 4$ and $k$ even. 
Put $a_1=\gcd(a+1,m+1)$ and $a_2=\gcd(a+1,2m+1)$. 
Put $$M=\frac{(m^2-1)(4m^2-1)}{6a_1a_2}.$$
Then

\noindent\textup{(a)} $m>a$;

\noindent\textup{(b)} $m-1$, $2m-1$, $(m+1)/a_1$, and $(2m+1)/a_2$ are all square-free;

\noindent\textup{(c)} if $p$ divides at least one of the above four integers, then $(p-1)\mid k$;

\noindent\textup{(d)} $m>3.4429\cdot 10^{82}$;

\noindent\textup{(e)} the number $M$ is square-free and has at least $139$ prime factors;

\noindent\textup{(f)} if $m\equiv 1\pmod{3}$, then $m>1.485\cdot 10^{9321155}$ and the number $M$ has at least\linebreak\hphantom{\textup{(f)}} $4990906$ prime factors;

\noindent\textup{(g)} if $m\equiv 1\pmod{30}$, then $m>10^{4\cdot 10^{20}}$ and the number $M$ has at least\linebreak\hphantom{\textup{(g)}} $75760524354901799895$ prime factors.

\end{Thm}

\begin{proof}
As $a$ is square-free, we have $a\mid m$, and so $m\ge a$. If $m=a$, then \eqref{eq12} yields
$$
\sum_{\substack{p\mid(m+1)\\(p-1)\mid k}}\frac 1p\equiv 0\pmod{1}.
$$
Since the sum of reciprocals of distinct primes can never be a positive integer, we must have
$$
\sum_{\substack{p\mid(m+1)\\(p-1)\mid k}}\frac 1p=0,
$$
which contradicts the fact that $2\mid(m+1)$. Parts~(b) and (c) are direct consequences of Lemmas~\ref{lem:l11} and \ref{lem:l12}. Further, using 
Lemma~\ref{lem:l12}, parts (b) and (c) and the facts that $a\mid m$ and $\gcd(a,m-1)=\gcd(a,2m-1)=\gcd(a,m+1)=\gcd(a,2m+1)=1$, we find that
\begin{align*}
 m^k&\equiv m\pmod{a(m-1)},\\
2m^k&\equiv 4m\pmod{a(2m-1)},\\
(a+1)m^k&\equiv a+1+(a-1)(m+1)\pmod{a(m+1)},\\
2(a+1)m^k&\equiv 2(a+1)+(a-2)(2m+1)\pmod{a(2m+1)}.
\end{align*}
Here we will only provide details for the latter congruence, which is the most complicated one to 
establish. Since $a$ and $2m+1$ are coprime, it suffices by the Chinese remainder theorem to establish
the congruence modulo $a$ and modulo $2m+1$. Since $a\mid{m}$ the congruence trivially holds modulo $a$.

Now suppose that $p\mid(2m+1)$. \\
{\tt First case}: $(p-1)\mid k$.\\
\noindent
By Lemma~\ref{lem:l12} we have
$\ord_p(2m+1)=\ord_p(a+1)+1$ and it suffices to show that $m^k\equiv 1\pmod{p}$. This
is true by Euler's theorem.\\
{\tt Second case}: $(p-1)\nmid k$.\\
Here we use that, by Lemma~\ref{lem:l12} again,
$\ord_p(2m+1)\le\ord_p(a+1)$ to see that the congruence holds.

We can rewrite the congruences \eqref{eq10}~--~\eqref{eq13} as
\begin{align}
\sum_{p\mid (m-1)}\frac 1p+\frac m{a(m-1)}&\equiv 0\pmod{1},\label{eq14}\\
\sum_{p\mid (2m-1)}\frac 1p+\frac {4m}{a(2m-1)}&\equiv 0\pmod{1},\label{eq15}\\
\sum_{p\mid\frac{m+1}{a_1}}\frac 1p+\frac{a+1+(a-1)(m+1)}{a(m+1)}&\equiv 0\pmod{1},\label{eq16}\\
\sum_{p\mid\frac{2m+1}{a_2}}\frac 1p+\frac{2(a+1)+(a-2)(2m+1)}{a(2m+1)}&\equiv 0\pmod{1}.\label{eq17}
\end{align}
By Corollary~\ref{cor:c3mod4}, the assumption that $k$ is even and
Lemma~\ref{lem:l12}, we see that
$(m+1)/a_1$ is even. Now noting that $a\ge 37$ (by Corollary~\ref{cor:c5} and the fact that $37$ is the first irregular prime), we have
$$
\sum_{p\mid\frac{m+1}{a_1}}\frac 1p+\frac{a+1+(a-1)(m+1)}{a(m+1)}\ge\frac 12+\frac{a+1+a(m+1)}{a(m+1)}-\frac 1a>1.
$$
Therefore, if we add the left hand sides of \eqref{eq14}, \eqref{eq15}, \eqref{eq16} and \eqref{eq17}, we get an integer, at least $5$. No prime $p>3$ can divide more than one of the integers $m-1$, $2m-1$, $(m+1)/a_1$, and $(2m+1)/a_2$, and $2$ and $3$ divide precisely two of these integers. Hence $M=(m^2-1)(4m^2-1)/(6a_1a_2)$ is square-free and
\begin{align}
\sum_{p\mid M}\frac 1p+\frac1{a(m-1)}&+\frac2{a(2m-1)}+\frac{a+1}{a(m+1)}+\frac{2(a+1)}{a(2m+1)}\notag\\
&\ge 3-\frac12-\frac13=2\,\frac16.\label{eq18}
\end{align}
Since $a\mid m$, $m>a\ge 37$ and each prime divisor of $m$ is irregular, we have $m\ge 37^2$. 
A simple computation shows that \eqref{eq14} is never satisfied for $a\ge 37$ and $m=37^2$. Since
$59$ is the second irregular prime, it follows that $m\ge 37\cdot 59$. 
On noting that the four fractions above are decreasing functions in both $a$ and $m$, we find
on substituting $a=37$ and $m=37\cdot 59$ that
$\sum_{p\mid M}\frac 1p>\alpha$, with $\alpha=2.1657$. Note that if
\begin{equation}
\label{eq19}
\sum_{p\le x}\frac 1p < \alpha,
\end{equation}
then $m^4/3>M>\prod_{p\le x}p$ (note that $a_1\ge 2$ and hence $M<m^4/3$). One computes (using a computer algebra package) the largest prime $p_s$ such 
that $\sum_{p_j\le p_s}\frac 1{p_j}<2\frac16$, with $p_1,p_2,\ldots$ the consecutive primes. Here one finds that $s=139$ and
$$
\sum_{j=1}^{139}\frac 1{p_j}<2.16566< \alpha< \sum_{p\mid M}\frac 1p. 
$$
Thus
$$
m>\biggl(3\prod_{p\le p_{139}}p\biggr)^{1/4} >3.4429\cdot 10^{82}.
$$

Now assume that $m\equiv 1\pmod{3}$. Then $3\mid{(2m+1)/a_1}$ by Lemma~\ref{lem:l12}, and so
$$
\sum_{p\mid\frac{2m+1}{a_2}}\frac 1p+\frac{2(a+1)+(a-2)(2m+1)}{a(2m+1)}\ge\frac 13+\frac{2(a+1)+a(2m+1)}{a(2m+1)}-\frac1{a/2}>1.
$$
Hence in this case, $2\frac 16$ in \eqref{eq18} can be replaced by $3\frac 16$. 
This $\alpha$ occurs in the work of Moser and here it is known
that $s=4990906$ leading to $m>1.485\cdot 10^{9321155}$ (cf.~\cite{BJM,tophat}).

Finally, assume that $m\equiv 1\pmod{30}$. Then
\begin{equation}
\label{laatsteloodjes}
\sum_{p\mid (m-1)}\frac 1p+\frac m{a(m-1)}>\frac 12+\frac 13+\frac 15>1,
\end{equation}
and we have the inequality~\eqref{eq18} with $2\frac 16$ replaced by $4\frac 16$. In this case
the largest prime $p_s$ such that $\sum_{p\le p_s}\frac{1}{p}<4\frac 16$ can no longer be determined by
direct computation and more sophisticated methods are needed, cf.~\cite{BKS, BBM}.
\end{proof}

\noindent {\tt Remark}. Note that in the proof we only used that $a\ge 37$. This has as a consequence that
the proof only depends on the first assertion in Theorem~\ref{main}.

\section{Proofs of the new results announced in the introduction}
\label{sec:proofs}

It remains to establish Theorem~\ref{main}, Theorem~\ref{main2} and Proposition~\ref{unique}.

\begin{proof}[Proof of Theorem~$\ref{main}$]
The first restriction on $a$ arises on invoking Corollary~\ref{cor:c5}.
In order to obtain the second restriction we have two write down all
integers $a\le 1500$ that are composed only of irregular primes.
These are listed in Table~\ref{table3}. Each of these can be excluded as is shown for
two examples in Section \ref{exclusive}. This exclusion process for 
each $a$ can be reconstructed using Table~\ref{table3}.
\end{proof}

To prove Theorem~\ref{main2} we need the following result, which shows that $m$ and $k$ are of comparable size.
\begin{Lem}
\label{lem:l13}
Suppose that $aS_k(m)=m^k$. Then $k+1<am<(a+1)(k+1)$.
\end{Lem}

\begin{proof}
We have
$$
S_k(m)\le\int\limits_1^m x^k\d x\le S_k(m+1)-1.
$$
Hence
$$
m^k=aS_k(m)\le a\int\limits_1^m x^k\d x=\frac{a(m^{k+1}-1)}{k+1}<\frac{am^{k+1}}{k+1},
$$
and so $am>k+1$. Further,
$$
(a+1)m^k=aS_k(m+1)\ge a\biggl(1+\int\limits_1^m x^k\d x\biggr)=\frac{a(m^{k+1}+k)}{k+1}>\frac{am^{k+1}}{k+1},
$$
that is $am<(a+1)(k+1)$.
\end{proof}

\begin{proof}[Proof of Theorem~$\ref{main2}$]
The lower bound on $m$ is a consequence of Theorem~\ref{thm:t6}, part (d).
On invoking Lemma~\ref{lem:l13} with $a\ge 1501$ the result
follows.
\end{proof}

\begin{proof}[Proof of Proposition~$\ref{unique}$]
Let $(m,k)$ be a solution of \eqref{KEME}. Observe that
$$
\frac{S_j(m)}{m^j}=\Bigl(\frac 1m\Bigr)^j+\Bigl(\frac 2m\Bigr)^j+\dots+\Bigl(\frac{m-1}m\Bigr)^j>\frac{S_{j+1}(m)}{m^{j+1}},
$$
and so
\begin{equation}
\label{eq20}
\begin{split}
aS_j(m)-m^j>0&\qquad\qquad\text{if $j<k$},\\
aS_j(m)-m^j<0&\qquad\qquad\text{if $j>k$}.
\end{split}
\end{equation}
This shows that for every $m$, there is at most one $k$.\\
\indent Now assume that there exists a positive integer $n$ such that $aS_k(m+n)=(m+n)^k$. Then $k>1$. Since
$$
S_k(m+n)=S_k(n)+\sum_{j=0}^{m-1}(n+j)^k=S_k(n)+mn^k+\sum_{j=1}^k \binom kj n^{k-j}S_j(m),
$$
we have
$$
(m+n)^k=\sum_{j=0}^k \binom kj m^jn^{k-j}=aS_k(n)+amn^k+a\sum_{j=1}^k \binom kj n^{k-j}S_j(m),
$$
or, equivalently,
$$
n^k=aS_k(n)+amn^k+\sum_{j=1}^{k-1} \binom kj n^{k-j}(aS_j(m)-m^j).
$$
In view of \eqref{eq20}, the last equality cannot hold. Thus we see that, for every $k$, there is at most
one $m$.
\end{proof}

\noindent {\tt Remark}. Using the same type of argument, we can prove the following: if $(m_1,k_1)$ and $(m_2,k_2)$ are two distinct solutions of $aS_k(m)=m^k$ then either $m_1>m_2$, $k_1>k_2$ or $m_1<m_2$, $k_1<k_2$.

\section{Other properties of Kellner-Erd\H{o}s-Moser solutions}
\label{sec:other}
There are many restrictions known that a solution of the Erd\H{o}s-Moser equation has to satisfy. We expect
that most of these have an analogue for the  Kellner-Erd\H{o}s-Moser equation as well. We present
an example.
\begin{Thm}
\label{thm:t8}
Suppose that $aS_k(m)=m^k$ with $m\ge 4$ and $k$ even. 
We have $$\ord_2(a(m-1)-2)\begin{cases}
=3+\ord_2 k & \text{if  $a\equiv 1\!\!\pmod{4}$,}\\
\ge 4+\ord_2 k & \text{if $a\equiv 3\!\!\pmod{4}$.}
\end{cases}
$$
\end{Thm}
For $a=1$ this was first established in Moree et al.~\cite[Lemma~12]{MRU} using Bernoulli
numbers. A much more elementary proof also dealing with solutions of the equation $S_k(m)=am^k$ 
for integers $a$ was given later by the second author \cite{Oz}.\\
\indent Our proof of Theorem~\ref{thm:t8} makes use of the following lemma.
\begin{Lem}
\label{power2}
Let $k$ and $m$ be positive integers where $k\ge 6$ is even. Then
$$
S_k(m)\equiv\begin{cases}
\left[\frac m2\right]+2^{2+\ord_2 k}\pmod{2^{3+\ord_2 k}}&\text{if $\left[\frac m2\right]\equiv 2\pmod{4}$,}\\
\left[\frac m2\right]\pmod{2^{3+\ord_2 k}}&\text{otherwise.}
\end{cases}
$$
\end{Lem}

\begin{proof}
It is easily proved by induction on $s$ that for an odd integer $j$ and $s\ge 1$
$$
j^{2^s}\equiv\begin{cases}
1\pmod{2^{s+3}}&\text{if $j\equiv\pm 1\pmod{8}$,}\\
2^{s+2}+1\pmod{2^{s+3}}&\text{if $j\equiv\pm 3\pmod{8}$.}
\end{cases}
$$
Note that $k\ge 3+\ord_2 k$. Indeed, for $\ord_2 k=1$ and $\ord_2 k=2$ it follows from the condition $k\ge 6$ and for $\ord_2 k\ge 3$ we have $k\ge 2^{\ord_2 k}\ge 3+\ord_2 k$.
Thus for an integer $j$ we have
$$
j^k\equiv\begin{cases}
0\pmod{2^{3+\ord_2 k}}&\text{if $j$ is even,}\\
1\pmod{2^{3+\ord_2 k}}&\text{if $j\equiv\pm 1\pmod{8}$,}\\
2^{2+\ord_2 k}+1\pmod{2^{3+\ord_2 k}}&\text{if $j\equiv\pm 3\pmod{8}$.}
\end{cases}
$$

Assume that $m\equiv 1\pmod{4}$. Then
$$
\#\{1\le n< m\,:\,n\equiv\pm 1\!\!\pmod{8}\}=\#\{1\le n<m\,:\,n\equiv\pm 3\!\!\pmod{8}\}=\frac{m-1}4.
$$
Hence
\begin{align*}
S_k(m)&\equiv \frac{m-1}4+\frac{m-1}4\,(2^{2+\ord_2 k}+1)\equiv\frac{m-1}2\,(2^{1+\ord_2 k}+1)\\
&\equiv\begin{cases}
\frac{m-1}2\pmod{2^{3+\ord_2 k}}&\text{if $m\equiv 1\pmod{8}$,}\\
\frac{m-1}2+2^{2+\ord_2 k}\pmod{2^{3+\ord_2 k}}&\text{if $m\equiv 5\pmod{8}$.}
\end{cases}
\end{align*}
Now assume that $m\equiv 3\pmod{4}$. Then
$$
\#\{1\le n\le m\,:\,n\equiv\pm 1\!\!\pmod{8}\}=\#\{1\le n\le m\,:\,n\equiv\pm 3\!\!\pmod{8}\}=\frac{m+1}4.
$$
This yields
\begin{align*}
S_k(m+1)&\equiv \frac{m+1}4+\frac{m+1}4\,(2^{2+\ord_2 k}+1)\equiv\frac{m+1}2\,(2^{1+\ord_2 k}+1)\\
&\equiv\begin{cases}
\frac{m+1}2\pmod{2^{3+\ord_2 k}}&\text{if $m\equiv 7\pmod{8}$,}\\
\frac{m+1}2+2^{2+\ord_2 k}\pmod{2^{3+\ord_2 k}}&\text{if $m\equiv 3\pmod{8}$.}
\end{cases}
\end{align*}
Since
$$
m^k\equiv\begin{cases}
1\pmod{2^{3+\ord_2 k}}&\text{if $m\equiv 7\pmod{8}$,}\\
2^{2+\ord_2 k}+1\pmod{2^{3+\ord_2 k}}&\text{if $m\equiv 3\pmod{8}$,}
\end{cases}
$$
we deduce that
$$
S_k(m)=S_k(m+1)-m^k\equiv\frac{m-1}2\pmod{2^{3+\ord_2 k}}.
$$
Finally, if $m$ is even then
$$
S_k(m)\equiv S_k(m+1)\equiv\begin{cases}
\frac m2\pmod{2^{3+\ord_2 k}}&\text{if $m\not\equiv 4\pmod{8}$,}\\
\frac m2+2^{2+\ord_2 k}\pmod{2^{3+\ord_2 k}}&\text{if $m\equiv 4\pmod{8}$.}
\end{cases}
$$
This completes the proof.
\end{proof}

\begin{proof}[Proof of Theorem~$\ref{thm:t8}$]
By Lemma~\ref{power2} and Corollaries~\ref{cor:c3} and \ref{cor:c3mod4},
$$
S_k(m)\equiv\frac{m-1}2\pmod{2^{3+\ord_2 k}}.
$$
From Lemma~\ref{lem:l12} we see that
$$
m\equiv\begin{cases}
7\!\!\pmod{8}&\text{if $a\equiv 3\pmod{4}$,}\\
3\!\!\pmod{8}&\text{if $a\equiv 1\pmod{4}$.}
\end{cases}
$$
Therefore
$$
aS_k(m)\equiv a\cdot\frac{m-1}2\equiv m^k\equiv\begin{cases}
1\pmod{2^{3+\ord_2 k}}&\text{if $a\equiv 3\pmod{4}$,}\\
2^{2+\ord_2 k}+1\pmod{2^{3+\ord_2 k}}&\text{if $a\equiv 1\pmod{4}$,}
\end{cases}
$$
and so
$$
a(m-1)-2\equiv\begin{cases}
0\pmod{2^{4+\ord_2 k}}&\text{if $a\equiv 3\pmod{4}$,}\\
2^{3+\ord_2 k}\pmod{2^{4+\ord_2 k}}&\text{if $a\equiv 1\pmod{4}$,}
\end{cases}
$$
as desired.
\end{proof}

\begin{Thm}
\label{thm:t11}
Let $a>3$. If $aS_k(m)=m^k$ and $m$ is a prime, then $a=q^{2s}$ for some irregular 
prime~$q\equiv 3\pmod{16}$ and positive integer~$s$.
\end{Thm}
\begin{proof}
If $a$ has at least two distinct prime divisors, then $m$ cannot be a prime. Now assume that $a$ is a power of a prime $q$. Then $q$ is an irregular prime and $m=q$. Suppose that $a=q^{2s+1}$ for some $s\ge 0$. Then $(a+1)m^k/a(m+1)=(q^{2s+1}+1)S_k(q)/(q+1)$ is an integer, and \eqref{eq12} implies
$$
\sum_{\substack{p\mid(q+1)\\(p-1)\mid k}}\frac 1p\equiv 0\pmod{1}.
$$
Exactly as in the proof of Theorem~\ref{thm:t6}, we conclude that this is impossible. Finally, assume that $a=q^{2s}$ for some $s\ge 1$. Then $a\equiv 1\pmod{8}$. Note that $k$ has to be even. By Theorem~\ref{thm:t8} it follows that $a(q-1)\equiv 2\pmod{16}$, which yields $\frac{q-1}2\equiv 1\pmod{8}$, that is, $q\equiv 3\pmod{16}$. 
\end{proof}
\noindent {\tt Remark}. It is not known whether there are infinitely many irregular primes 
$q\equiv 3\pmod{16}$. However, from a result of Mets\"ankyl\"a~\cite{Metzger} it follows that
there are infinitely many irregular primes $q\equiv \pm 3, \pm 5\pmod{16}$.\\

\noindent {\tt Acknowledgement}. This paper was begun during the stay of the
first author in February-April 2014 at the Max Planck Institute for Mathematics. She likes to thank 
for the invitation and the pleasant research atmosphere. The second author
was introduced to the subject around 1990 by the late Jerzy Urbanowicz. He will never forget
his interest, help and kindness. Further he thanks Prof. T.~N.~Shorey for helpful discussions in the 
summer of 2014.
We like to heartily thank Bernd Kellner for 
helpful e-mail correspondence and patiently answering our Bernoulli number questions. A.~Ciolan, P.~Tegelaar and W.~Zudilin kindly
commented on an earlier version of this paper.\\
\indent The cooperation of the authors has its origin in them having met, in 2012, at ELAZ in Schloss Schney (which
was also attended by Prof. W.~Schwarz). This contribution is
our tribute to Prof. W.~Schwarz, who was one of the initiators of the ELAZ conference series.

\section*{Appendix}

\begin{center}
\begin{longtable}{|c|l|}
\caption{Pairs of irregular primes $(p_1,p_2)$ with $p_1<p_2$, $p_1p_2<50000$, satisfying the conditions of Corollary~\ref{cor:c6}}\\
\hline
$p_1$ & \hphantom{67, 103, 149, 157, 271, 307, 379,} $p_2$ \\
\hline\endhead\label{table1}
37 & 67, 103, 149, 157, 271, 307, 379, 401, 409, 421, 433, 463, 523, 541, 547,\\
&  557, 577, 593, 607, 613, 631, 673, 727, 757, 811, 877, 881, 1061, 1117,\\
&  1129, 1153, 1193, 1201, 1237, 1297, 1327\\
\hline
59 & 233, 523 
\\
\hline
67 & 103, 157, 271, 283, 409, 421, 433, 463, 541, 547, 613, 617, 619, 683, 691,\\
& 727\\
\hline
101 & 131, 149, 157, 271, 311, 401, 409, 421, 433, 461\\
\hline
103 & 157, 283, 307, 463 \\
\hline
131 & 157, 271\\
\hline
149 & 233, 257, 293\\
\hline
157 & 233, 257, 271, 293, 307\\
\hline
\end{longtable}
\end{center}

\begin{center}
\begin{longtable}{|c|c|l|}
\caption{Helpful pairs $(t,q)_a$ with $q\le 17$}\\
\hline
$q$ & $t$ & $\hskip60pt a\pmod{q}$\\
\hline\endhead\label{table2}
5 & 2 & 1, 2, 3\\
\hline
7 & 2 & 1, 3, 5\\
\hline
7 & 4 & 1, 3, 5\\
\hline
11 & 2 & 1, 2, 3, 5, 6, 7 \\
\hline
11 & 4 & 2, 3, 6, 7, 9 \\
\hline
11 & 6 & 1, 5, 6, 7 \\
\hline
11 & 8 & 6, 7, 9\\
\hline
13 & 2 & 1, 2, 5, 6, 8, 10, 11 \\
\hline
13 & 4 & 1, 2, 6, 8, 11 \\
\hline
13 & 6 & 2, 3, 4, 5, 7, 8, 9, 10, 11 \\
\hline
13 & 8 & 1, 2, 6, 8, 11 \\
\hline
13 & 10 & 1, 2, 4, 6, 7, 8, 11\\
\hline
17 & 2 & 1, 2, 5, 7, 8, 10, 13, 14 \\
\hline
17 & 4 & 1, 2, 3, 5, 8, 9, 11, 12, 15 \\
\hline
17 & 6 & 1, 2, 4, 5, 6, 8 \\
\hline
17 & 8 & 2, 3, 4, 5, 6, 7, 9, 10, 11, 12, 13, 14, 15\\
\hline
17 & 10 & 3, 7, 9, 11, 12, 14, 15 \\
\hline
17 & 12 & 1, 2, 3, 5, 7, 8, 12, 14, 15 \\
\hline
17 & 14 & 3, 4, 6, 7, 9, 11, 12, 14, 15 \\
\hline

\end{longtable}
\end{center}

\begin{center}
\begin{longtable}{|c|c|c|l|}
\caption{Irregular pairs $(r,p)$ along with the corresponding helpful pairs $(t_j,q_j)_{a\!\!\pmod{q_j}}$ satisfying the conditions of Proposition~\ref{prop:t5}}\\
\hline
$a$ & irregular pair $(r,p)$ & $\ell_r$ & \qquad\qquad helpful pairs $(t_j,q_j)_{a\!\!\pmod{q_j}}$\\
\hline\endhead\label{table3}
37 & $(32,37)$ & 1 & $(8,13)_{11}$\\
\hline
59 & $(44,59)$ & 6 & $(2,7)_3, (4,7)_3,(6,13)_7,(276,349)_{59}$\\
\hline
67 & $(58,67)$ & 2 & $(4,13)_2,(10,13)_2$\\
\hline
101 & $(68,101)$ & 6 & $(2,7)_3,(4,7)_3,(28,41)_{19},(168,601)_{101}$\\
\hline
103 & $(24,103)$ & 14 & $(2,5)_3,(16,29)_{16},(24,43)_{17},(92,239)_{103},$ \\
&&& $(194,239)_{103},(228,239)_{103},(636,1429)_{103},$\\
&&& $(840,1429)_{103}$\\
\hline
131 & $(22,131)$ & 4 & $(2,5)_1,(12,41)_8,(152,521)_{131}$\\
\hline
149 & $(130,149)$ & 15 & $(2,11)_6$, $(4,11)_6$, $(6,11)_6$, $(8,11)_6$, $(2,13)_6$, \\
&&& $(10,13)_6$, $(30,61)_{27}$\\
\hline
157 & $(62,157)$ & 1 & $(2,5)_2$\\
    & $(110,157)$ & 1 & $(2,5)_2$\\
\hline
233 & $(84,233)$ & 1 & $(26,59)_{56}$\\
\hline
257 & $(164,257)$ & 1 & $(4,17)_2$\\
\hline
263 & $(100,263)$ & 30 & $(2,5)_3,(2,31)_{15},(4,31)_{15},(8,31)_{15},(14,31)_{15},$\\
&&& $(16,31)_{15},(22,31)_{15},(28,31)_{15},(40,61)_{19},$\\
&&& $(624,787)_{263},(2196,2621)_{263},(2720,3931)_{263}$\\
\hline
271 & $(84,271)$ & 1 & $(4,11)_7$\\
\hline
283 & $(20,283)$ & 1 & $(2,7)_3$\\
\hline
293 & $(156,293)$ & 30 & $(2,11)_7,(4,11)_7,(6,11)_7,(8,11)_7,(20,41)_6,$\\
&&& $(40,61)_{49},(740,877)_{293},(1032,1753)_{293}$\\
\hline
307 & $(88,307)$ & 1 & $(88,103)_{101}$\\
\hline
311 & $(292,311)$ & 1& $(2,11)_3$\\
\hline
347 & $(280,347)$ & 30 & $(2,5)_2,(2,11)_6,(4,11)_6,(6,11)_6,(8,11)_6,$\\
&&& $(10,31)_6,(20,61)_{42},(972,1039)_{347}$\\
\hline
353 & $(186,353)$ & 1 & $(2,5)_3$\\
    & $(300,353)$ & 9 & $(2,7)_3,(4,7)_3,(12,37)_{20},(24,37)_{20},(36,73)_{61}$\\
\hline
379 & $(100,379)$ & 1 & $(4,7)_1$\\
    & $(174,379)$ & 1 & $(48,127)_{99}$\\
\hline
389 & $(200,389)$ & 4 & $(4,17)_{15},(8,17)_{15},(12,17)_{15},(976,1553)_{389}$\\ 
\hline
401 & $(382,401)$ & 1 & $(2,5)_1$\\
\hline
409 & $(126,409)$ & 3 & $(6,19)_{10},(12,19)_{10},(18,37)_2$\\
\hline
421 & $(240,421)$ & 2 & $(20,41)_{11},(240,281)_{140}$\\
\hline
433 & $(366,433)$ & 1 & $(2,5)_3$\\
\hline
461 & $(196,461)$ & 1 & $(12,47)_{38}$\\
\hline
463 & $(130,463)$ & 1 & $(4,7)_1$\\
\hline
467 & $(94,467)$ & 3 & $(2,7)_5,(4,7)_5,(1026,1399)_{467}$\\
    & $(194,467)$ & 18 & $(2,5)_2,(2,7)_5,(4,7)_5,(12,37)_{23},(24,37)_{23},$\\
    &&& $(3456,8389)_{467}$\\
\hline
491 & $(292,491)$ & 1 & $(2,11)_7$\\
    & $(336,491)$ & 1 & $(6,11)_7$\\
    & $(338,491)$ & 1 & $(8,11)_7$\\
\hline
523 & $(400,523)$ & 1 & $(4,7)_5$\\
\hline
541 & $(86,541)$ & 1 & $(2,5)_1$\\
\hline
547 & $(270,547)$ & 1 & $(36,79)_{73}$\\
    & $(486,547)$ & 2 & $(2,5)_2,(96,157)_{76}$\\
\hline
557 & $(222,557)$ & 1 & $(2,5)_2$\\
\hline
577 & $(52,577)$ & 1 & $(4,7)_3$\\
\hline
587 & $(90,587)$ & 6 & $(2,13)_2,(4,13)_2,(6,13)_2,(8,13)_2,(10,13)_2,$\\
&&& $(90,1759)_{587}$\\
    & $(92,587)$ & 6 & $(2,13)_2,(4,13)_2,(6,13)_2,(8,13)_2,(10,13)_2,$\\
&&& $(678,1759)_{587}$\\
\hline
593 & $(22,593)$ & 1 & $(2,5)_3$\\
\hline
607 & $(592,607)$ & 1 & $(4,7)_5$\\
\hline
613 & $(522,613)$ & 1 & $(2,5)_3$\\
\hline
\newpage
617 & $(20,617)$ & 1 & $(20,29)_8$\\
    & $(174,617)$ & 1 & $(2,5)_2$\\
    & $(338,617)$ & 1 & $(2,5)_2$\\
\hline
619 & $(428,619)$ & 1 & $(2,7)_3$\\
\hline
631 & $(80,631)$ & 1 & $(2,7)_1$\\
    & $(226,631)$ & 1 & $(4,7)_1$\\
\hline
647 & $(236,647)$ & 6 & $(2,5)_2,(2,7)_3,(4,7)_3,(84,229)_{189}$\\
    & $(242,647)$ & 3 & $(2,7)_3,(4,7)_3,(72,103)_{29}$\\
    & $(554,647)$ & 12 & $(2,5)_2,(2,7)_3,(4,7)_3,(180,409)_{238},$\\
    &&& $(288,457)_{190}$\\
\hline
653 & $(48,653)$ & 60 & $(8,17)_7,(12,17)_7,(4,31)_2,(14,31)_2,(12,41)_{38},$\\
&&& $(20,41)_{38},(28,41)_{38},(24,61)_{43},(36,61)_{43},$\\
&&& $(196,241)_{171},(48,2609)_{653},(9176,9781)_{653}$\\
\hline
659 & $(224,659)$ & 18 & $(2,7)_1,(4,7)_1,(6,13)_9,(12,19)_{13},(42,127)_{24},$\\
&&& $(882,5923)_{659}$\\
\hline
673 & $(408,673)$ & 1 & $(8,17)_{10}$\\
    & $(502,673)$ & 1 & $(2,5)_3$\\
\hline
677 & $(628,677)$ & 3 & $(4,13)_1,(8,13)_1,(30,79)_{45}$\\
\hline
683 & $(32,683)$ & 12 & $(2,5)_3,(32,67)_{13},(76,89)_{60},(280,373)_{310},$\\
&&& $(2760,4093)_{683}$\\
\hline
691 & $(12,691)$ & 1 & $(12,31)_9$\\
    & $(200,691)$ & 1 & $(2,7)_5$\\
\hline
727 & $(378,727)$ & 1 & $(4,23)_{14}$\\
\hline
751 & $(290,751)$ & 1 & $(40,251)_{249}$\\
\hline
757 & $(514,757)$ & 1 & $(2,5)_2$\\
\hline
761 & $(260,761)$ & 1 & $(20,41)_{23}$\\
\hline
773 & $(732,773)$ & 60 & $(2,7)_3,(4,7)_3,(2,11)_3,(4,11)_3,(4,17)_8,$\\
&&& $(12,17)_8,(16,41)_{35},(120,241)_{50},$\\
&&& $(1118,1931)_{773},(1504,3089)_{773}$\\
\hline
797 & $(220,797)$ & 9 & $(8,37)_{20},(20,37)_{20},(220,2389)_{797},$\\
&&& $(1812,2389)_{797},(2210,3583)_{797}$\\
\hline
809 & $(330,809)$ & 5 & $(2,11)_6,(4,11)_6,(6,11)_6,(8,11)_6,(10,41)_{30}$\\
    & $(628,809)$ & 5 & $(2,11)_6,(4,11)_6,(6,11)_6,(8,11)_6,(20,41)_{30}$\\
\hline
811 & $(544,811)$ & 1 & $(4,19)_{13}$\\
\hline
821 & $(744,821)$ & 1 & $(4,11)_7$\\
\hline
827 & $(102,827)$ & 6 & $(2,13)_8, (4,13)_8, (6,13)_8, (8,13)_8, (10,13)_8,$\\
&&& $(2580,4957)_{827}$\\
\hline
\newpage
839 & $(66,839)$ & 1680 & $(2,11)_3,(4,11)_3,(6,13)_7,(10,13)_7,(8,17)_6,$\\
&&& $(2,29)_{27},(4,29)_{27},(8,29)_{27},(10,29)_{27},$\\
&&& $(12,29)_{27},(14,29)_{27},(16,29)_{27},(18,29)_{27},$ \\
&&& $(20,29)_{27},(22,29)_{27},(10,31)_2,(20,31)_2,$ \\
&&& $(10,41)_{19},(16,41)_{19},(20,41)_{19},(26,41)_{19},$\\
&&& $(28,41)_{19},(30,41)_{19},(38,41)_{19},(28,61)_{46},$ \\
&&& $(48,61)_{46},(56,61)_{46},(28,71)_{58},(56,71)_{58},$ \\
&&& $(4,97)_{63},(52,97)_{63},(76,97)_{63},(26,211)_{206},$ \\
&&& $(66,211)_{206},(6,281)_{277},(80,281)_{277},$\\
&&& $(138,281)_{277},(248,281)_{277},(258,281)_{277},$\\
&&& $(220,673)_{166},(17664,20113)_{839}$\\
\hline
877 & $(868,877)$ & 1 & $(4,13)_6$\\
\hline
881 & $(162,881)$ & 1 & $(2,5)_1$\\
\hline
887 & $(418,887)$ & 3 & $(2,7)_5,(4,7)_5,(2190,2659)_{887}$\\
\hline
929 & $(520,929)$ & 1 & $(8,17)_{11}$\\
    & $(820,929)$ & 1 & $(4,17)_{11}$\\
\hline
953 & $(156,953)$ & 1 & $(20,137)_{131}$\\
\hline
971 & $(166,971)$ & 42 & $(2,7)_5,(4,7)_5,(2,29)_{14},(16,29)_{14},(18,43)_{25},$\\
&&& $(24,43)_{25},(36,43)_{25},(26,71)_{48},(56,71)_{48},$\\
&&& $(6,211)_{127}$\\  
\hline
1061 & $(474,1061)$ & 1 & $(2,5)_1$\\
\hline
1091 & $(888,1091)$ & 8 & $(2,5)_1,(4,17)_3,(8,17)_3,(12,17)_3,(8,41)_{25}$\\
\hline
1117 & $(794,1117)$ & 1 & $(2,5)_2$\\
\hline
1129 & $(348,1129)$ & 5 & $(2,11)_7,(4,11)_7,(6,11)_7,(8,11)_7,(20,41)_{22}$\\
\hline
1151 & $(534,1151)$ & 1 & $(4,11)_7$\\
     & $(784,1151)$ & 1 & $(4,11)_7$\\
     & $(968,1151)$ & 1 & $(8,11)_7$\\
\hline
1153 & $(802,1153)$ & 1 & $(2,5)_3$\\
\hline
1193 & $(262,1193)$ & 1 & $(2,5)_3$\\
\hline
1201 & $(676,1201)$ & 1 & $(4,17)_{11}$\\
\hline
1217 & $(784,1217)$ & 3 & $(4,13)_8,(8,13)_8,(48,97)_{53}$\\
& $(866,1217)$ & 1 & $(2,17)_{10}$\\
& $(1118,1217)$ & 3 & $(2,13)_8,(6,13)_8,(10,13)_8$\\
\hline
1229 & $(784,1229)$ & 28 & $(4,17)_5,(8,17)_5,(12,17)_5,(16,29)_{11},$\\
&&& $(32,113)_{99},(784,8597)_{1229},(2012,8597)_{1229},$\\
&&& $(3240,8597)_{1229},(5696,8597)_{1229},$\\
&&& $(6924,8597)_{1229}$\\
\hline
1237 & $(874,1237)$ & 1 & $(2,5)_2$\\ 
\hline
1279 & $(518,1279)$ & 1 & $(2,7)_5$\\
\hline
1283 & $(510,1283)$ & 10 & $(2,5)_3,(2,11)_7,(4,11)_7,(6,11)_7,(8,11)_7,$\\
&&& $(6920,12821)_{1283}$ \\
\hline
1291 & $(206,1291)$ & 2 & $(2,5)_1,(56,61)_{10}$\\
& $(824,1291)$ & 1 & $(14,31)_{20}$\\
\hline
1297 & $(202,1297)$ & 1 & $(2,5)_2$\\
& $(220,1297)$ & 1 & $(12,17)_5$\\
\hline
1301 & $(176,1301)$ & 1 & $(20,53)_{29}$\\
\hline
1307 & $(382,1307)$ & 3 & $(2,7)_5,(4,7)_5,(2994,3919)_{1307}$\\
& $(852,1307)$ & 3 & $(2,7)_5,(4,7)_5,(852,3919)_{1307}$\\
\hline
1319 & $(304,1319)$ & 360 & $(2,7)_3,(4,7)_3,(8,17)_{10},(6,31)_{17},(12,31)_{17},$\\
&&& $(18,31)_{17},(24,31)_{17},(6,37)_{24},(12,37)_{24},$\\
&&& $(18,37)_{24},(30,73)_5,(36,73)_5,(66,73)_5,$\\
&&& $(2940,11863)_{1319},(42480,52721)_{1319}$\\
\hline
1327 & $(466,1327)$ & 2 & $(2,5)_2,(4,13)_1$\\
\hline
1367 & $(234,1367)$ & 6 & $(2,5)_2,(4,13)_2,(8,13)_2,(234,4099)_{1367}$\\
\hline
1369 & $(32,37)$ & 10 & $(2,11)_5,(6,11)_5,(14,31)_5,(8,41)_{16},(20,61)_{27}$\\
\hline
1381 & $(266,1381)$ & 1 & $(2,5)_1$\\
\hline
1409 & $(358,1409)$ & 1 & $(6,23)_6$\\
\hline
1429 & $(996,1429)$ & 1 & $(44,239)_{234}$\\
\hline
1439 & $(574,1439)$ & 630 & $(4,11)_9,(8,11)_9,(2,19)_{14},(4,19)_{14},(8,19)_{14},$\\
&&& $(10,19)_{14},(12,19)_{14},(2,31)_{13},(6,31)_{13},$\\
&&& $(12,31)_{13},(22,31)_{13},(26,31)_{13},(24,37)_{33},$\\
&&& $(6,43)_{20},(24,43)_{20},(10,61)_{36},(30,61)_{36},$ \\
&&& $(40,61)_{36},(30,71)_{19},(40,71)_{19},(36,127)_{42},$\\
&&& $(16,181)_{172},(50,181)_{172},(140,181)_{172},$\\
&&& $(60,211)_{173},(574,8629)_{1439},(50904,60397)_{1439}$ \\
\hline
1483 & $(224,1483)$ & 2 & $(2,13)_1,(8,13)_1$\\
\hline
1499 & $(94,1499)$ & 24 & $(2,7)_1,(4,7)_1,(6,13)_4,(4,17)_3,(8,17)_3,$\\
&&& $(10,17)_3,(12,17)_3,(14,17)_3,(736,857)_{642}$\\
\hline
\end{longtable}
\end{center}

\end{document}